\documentclass[12pt]{article}
\usepackage[latin1]{inputenc}
\usepackage[british]{babel}
\usepackage{lmodern}
\usepackage[T1]{fontenc}
\usepackage[paper=a4paper, left=28mm, right=25mm, top=29mm, bottom=25mm]{geometry}
\usepackage{latexsym,amsfonts,amsmath,graphics}
\usepackage{epsfig}
\usepackage{enumitem}
\usepackage{amssymb,mathrsfs}
\usepackage{verbatim}
\newtheorem{theorem}{Theorem}
\newtheorem{lemma}{Lemma}
\newtheorem{corollary}{Corollary}

\newtheorem{remark}{Remark}
\newtheorem{definition}{Definition}
\newtheorem{algorithm}{Algorithm}
\newenvironment{proof}{\begin{trivlist} \item[\hskip\labelsep{\it Proof.}]}{$\hfill\Box$\end{trivlist}}

\newcommand{\rd}{\,\mathrm{d}}

\newcommand{\bsh}{\boldsymbol{h}}

\newcommand{\bsx}{\boldsymbol{x}}

\newcommand{\bsg}{\boldsymbol{g}}

\newcommand{\bst}{\boldsymbol{t}}

\newcommand{\bsw}{\boldsymbol{w}}
\newcommand{\bsgamma}{\boldsymbol{\gamma}}

\newcommand{\bszero}{\boldsymbol{0}}
\newcommand{\bsone}{\boldsymbol{1}}

\newcommand{\RR}{\mathbb{R}}

\newcommand{\back}{\backslash}
\newcommand{\bs}{\boldsymbol}

\newcommand{\FF}{\mathbb{F}}

\newcommand{\NN}{\mathbb{N}}
\newcommand{\PP}{\mathbb{P}}

\newcommand{\ii}{\mathrm{i}}

\newcommand{\ZZ}{\mathbb{Z}}

\newcommand{\cP}{\mathcal{P}}

\newcommand{\wal}{\mathrm{wal}}

\allowdisplaybreaks

\title{A reduced fast construction of polynomial lattice point sets with low weighted star discrepancy}
\author{Ralph Kritzinger\thanks{R. Kritzinger is supported by the Austrian Science Fund (FWF): Project F5509-N26, which is a part of the Special Research Program "Quasi-Monte Carlo Methods: Theory and Applications".} ,
Helene Laimer\thanks{H. Laimer is supported by the Austrian Science Fund (FWF): Project F5506-N26, which is a part of the Special Research Program "Quasi-Monte Carlo Methods: Theory and Applications".} ,
Mario Neum\"{u}ller\thanks{M. Neum\"{u}ller is supported by the Austrian Science Fund (FWF): Project F5505-N26, which is a part of the Special Research Program "Quasi-Monte Carlo Methods: Theory and Applications".}}
\date{}

\begin{document}

\maketitle

\begin{abstract}
\noindent
The weighted star discrepancy is a quantitative measure for the performance
of point sets in quasi-Monte Carlo algorithms for numerical integration.
We consider polynomial lattice point sets, whose generating vectors can be obtained by a
component-by-component construction to ensure a small weighted star discre-pancy. 
Our aim is to significantly reduce the construction cost of such generating vectors by restricting the size of
the set of polynomials from which we select the components of the vectors. To gain this reduction we exploit the fact that the weights of the spaces we consider decay very fast.
\end{abstract} 

\centerline{\begin{minipage}[hc]{130mm}{
{\em Keywords:} weighted star discrepancy, polynomial lattice point sets, quasi-Monte Carlo integration, component-by-component algorithm\\
{\em MSC 2000:} 11K06, 11K38, 65D30, 65D32}
\end{minipage}}

 \allowdisplaybreaks
 
\section{Introduction}

A convenient way to approximate the value of an integral $$ I_s(F):=\int_{[0,1)^s} F(\bsx)\rd\bsx $$ over the $s$-dimensional unit cube is to use a quasi-Monte Carlo rule of the form
  \begin{align}\label{eq:QMCkrilaineupoly}
	Q_{N,s}(F):=\frac{1}{N}\sum_{n=0}^{N-1}F(\bsx_n).
	\end{align}
The integrand $F$ usually stems from some suitable (weighted) function space and the multiset $\cP$ of integration nodes $\bsx_0,\bsx_1,\dots,\bsx_{N-1}$ in the algorithm $Q_{N,s}(F)$  is chosen deterministically from $[0,1)^s$.
For comprehensive information on quasi-Monte Carlo algorithms consult, e.g.,~\cite{DP, DKS13,LP,Nied92book}.
The quality of a quasi-Monte Carlo rule is for instance measured by some notion of discrepancy.
In this paper we consider the weighted star discrepancy, which has been introduced by Sloan and Wo\'{z}niakowski in \cite{SW}, exploiting the insight that the weights reflect the influence of different coordinates on the integration error. Let $[s]:=\{1,2,\dots,s\}$ and consider a weight sequence $\bsgamma=(\gamma_{\mathfrak{u}})_{\mathfrak{u}\subseteq [s]}$ of nonnegative real numbers, i.e., every group of variables $(x_i)_{i\in \mathfrak{u}}$ is equipped with a weight $\gamma_{\mathfrak{u}}$. Roughly speaking, a small weight indicates that the corresponding variables contribute little to the integration problem. For simplicity, throughout this paper we only consider product weights, defined as follows. Given a non-increasing sequence of positive real numbers $(\gamma_j)_{j\geq 1}$ with $\gamma_j\leq 1$ we set $\gamma_{\mathfrak{u}}:=\prod_{j\in\mathfrak{u}}\gamma_j$ and $\gamma_{\emptyset}:=1$.
\begin{definition}\label{StarDisckrilaineupoly}
Let $\bsgamma = (\gamma_{\mathfrak{u}})_{\mathfrak{u} \subseteq [s]}$ be a weight sequence and $\cP = \left\{ \bsx_0 , \dotsc , \bsx_{N-1} \right\} \subseteq [0,1)^s$ be an $N$-element point set. The local discrepancy of the point set $\cP$ at $\bst = (t_1, \dotsc, t_s) \in (0,1]^s$ is defined as
$$\Delta(\bst,\cP) := \frac{1}{N}\sum\limits_{n=0}^{N-1}{\bsone_{[\bszero,\bst)}(\bsx_n)} - \prod_{j=1}^s{t_j},$$
where $\bsone_{[\bszero,\bst)}$ denotes the characteristic function of $[\bszero,\bst) := [0,t_1)\times \cdots \times [0,t_s)$. The weighted star discrepancy of $\cP$ is then defined as
\begin{equation*}
D_{N, \bsgamma}^*(\cP) := \sup_{\bst \in (0,1]^s}{ \max_{\emptyset \neq \mathfrak{u} \subseteq [s]}{\gamma_{\mathfrak{u}}|\Delta((\bst_{\mathfrak{u}},\bsone),\cP)|}},
\end{equation*}
where $(\bst_{\mathfrak{u}},\bsone)$ denotes the vector $(\tilde{t}_1, \dotsc , \tilde{t}_s)$ with $\tilde{t}_j = t_j$ if $j \in \mathfrak{u}$ and $\tilde{t}_j = 1$ if $j \notin \mathfrak{u}$.
\end{definition} 
A relation between the integration error of quasi-Monte Carlo rules and the weighted star discrepancy is given by the Koksma-Hlawka type inequality (see~\cite{SW})
$$ |Q_{N,s}(F)-I_s(F)|\leq D_{N,\bsgamma}^{\ast}(\cP)\|F\|_{\bsgamma}, $$
where $\|\cdot\|_{\bsgamma}$ is some norm which depends only on the weight sequence $\bsgamma$ but not on the point set $\cP$. \vspace{3mm}

It turns out that lattice point sets (see, e.g., \cite[Chapter 5]{Nied92book}, \cite{LP}) and polynomial lattice point sets (see, e.g., \cite[Chapter 4]{Nied92book}, \cite{Nied92paper}, \cite[Chapter 10]{DP}) are often a good choice as sample points in \eqref{eq:QMCkrilaineupoly}. These two kind of point sets are strongly connected and have a lot of parallel tracks in their analysis. However, there are some situations were one type of point set is superior to the other in terms of error bounds or the size of the function classes where they yield good results for numerical integration. Thus it is beneficial to have constructions for lattice point sets as well as for polynomial lattice point sets at hand. For a detailed comparison of lattice point sets and polynomial lattice point sets see, e.g., \cite{P12}. Also, Ch. Schwab, in response to the first author's talk about constructing lattice point sets at the MCQMC 2016 conference in Stanford, pointed out that it would be an interesting problem to extend the result in \cite{KL15} to polynomial lattice point sets. Thus, in this paper we study polynomial lattice point sets, a special class of point sets with low weighted star discrepancy, introduced by Niederreiter in \cite[Chapter 4]{Nied92book}, \cite{Nied92paper}. For a prime number $p$, let $\FF_p$ be the finite field of order $p$. We identify $\FF_p$ with the set $\{0,1,\ldots,p-1\}$ equipped with the modulo $p$ arithmetic. We denote by $\FF_p[x]$ the set of polynomials over $\FF_p$ and by $\FF_p((x^{-1}))$ the field of formal Laurent series over $\FF_p$ with elements of the form
$$ L=\sum_{l=\omega}^{\infty}t_lx^{-l}, $$
where $\omega\in\ZZ$ and $t_l\in\FF_p$ for all $l\geq \omega$. For a given dimension $s\geq 2$ and some integer $m\geq 1$ we choose a so-called modulus $f\in \FF_p[x]$ with $\deg(f)=m$ as well as polynomials $g_1,\dots,g_s \in \FF_p[x]$. The vector $\bsg=(g_1,\dots,g_s)$ is called the generating vector of the polynomial lattice point set. Further, we introduce the map $\phi_m: \FF_p((x^{-1})) \to [0,1)$ such that
$$ \phi_m\left(\sum_{l=\omega}^{\infty}t_lx^{-l}\right)=\sum_{l=\max\{1,\omega\}}^{m}t_lp^{-l}. $$
With $n\in \{0,1,\dots,p^m-1\}$ we associate the polynomial
$$ n(x)=\sum_{r=0}^{m-1}n_rx^r \in \FF_p[x], $$ 
as each such $n$ can uniquely be written as $n=n_0+n_1p+\dots+n_{m-1}p^{m-1}$ with digits $n_r\in\{0,1,\dots,p-1\}$ for all $r\in\{0,1,\dots,m-1\}$. 
With this notation, the polynomial lattice point set $\cP(\bsg,f)$ is defined as the set of $N:=p^m$ points
$$ \bsx_n=\left(\phi_m\left(\frac{n(x)g_1(x)}{f(x)}\right),\dots,\phi_m\left(\frac{n(x)g_s(x)}{f(x)}\right)\right)\in [0,1)^s $$
for $0\leq n \leq p^m-1$. See also \cite[Chapter 10]{DP}.

In the following, by $G_{p,m}$ we denote the set of all polynomials $g$ over $\FF_p$ with $\deg(g)<m$. Further we define
\begin{align}\label{eq:gpmkrilaineupoly}
 G_{p,m}(f):=\{g\in G_{p,m} \mid \gcd(g,f)=1\}.
\end{align}
For the weighted star discrepancy of a polynomial lattice point set we simply write $D_{N,\bsgamma}^{\ast}(\bsg,f)$. 

Niederreiter~\cite{Nied92book} proved the existence of polynomial lattice point sets with low unweighted star discrepancy by averaging arguments. Generating vectors of good polynomial lattice point sets can be constructed by a component-by-component (CBC) construction. The standard structure of CBC constructions is as follows. We start by setting the first coordinate of the generating vector equal to 1. After this first step we proceed by increasing the dimension of the generating vector by one in each step until we have a generating vector $(g_1, \dots, g_s)$ of full size $s$. That is, all previously chosen components stay the same and one new component is added. This new coordinate is chosen from a given search set, most commonly from $G_{p,m}(f)$ given by \eqref{eq:gpmkrilaineupoly}. Usually it is determined such that the weighted star discrepancy of the lattice point set, corresponding to the generating vector, consisting of all previously chosen components plus one additional component, is minimized as a function of this last component.

Such constructions were provided in~\cite{DLP05} for an irreducible modulus $f$ and in~\cite{DKLP07} for a reducible $f$. In these papers, the authors considered the unweighted star discrepancy as well as its weighted version, which we study here. It is the aim of the present paper to speed up these constructions by reducing the search sets for the components of the generating vector $\bsg$ according to each component's importance. It is the nature of product weighted spaces that the components $g_j$ of the generating vector have less and less influence on the quality of the corresponding polynomial lattice point as $j$ increases.
 Roughly speaking this is due to the weights $(\gamma_j)$ that are becoming ever smaller with increasing index $j$. We want to exploit this property in the following way. As the components' influence is decreasing with their indices we want to use less and less time and computational cost to choose these components. To achieve this we choose them from smaller and smaller search sets, which are defined as follows. Let $w_1\leq w_2 \leq \cdots$ be a non-decreasing sequence of nonnegative integers. This sequence of $w_j$'s is determined in accordance with the weight sequence $\bsgamma$. Loosely speaking, the smaller $\gamma_j$, the bigger $w_j$ is chosen. For $w\in\NN_0$ with $w<m$ we define $G_{p,m-w}$ and $G_{p,m-w}(f)$ analogously to $G_{p,m}$ and $G_{p,m}(f)$, respectively. Further we set
$$ \mathcal{G}_{p,m-w}(f):=\begin{cases}
                    G_{p,m-w}(f) & \mbox{if\,} w<m, \\
										\{1\in\FF_p[x]\} & \mbox{if\,} w\geq m
                 \end{cases} $$
for any $w\in\NN_0$. For $w<m$ these sets have cardinality $p^{m-w}-1$ in the case of an irreducible modulus $f$ and $p^{m-w-1}(p-1)$ for the special case $f: \FF_p\to\FF_p, x \mapsto x^m$. We will consider these two cases in what follows.
Finally, for $d\in [s]$, we define $\mathcal{G}_{p,m-\bsw}^d(f):=\mathcal{G}_{p,m-w_1}(f)\times \dots \times \mathcal{G}_{p,m-w_d}(f)$.
 The idea is to choose the $i$th component of $\bsg$ of the form $x^{w_i}g_i$, where $g_i\in \mathcal{G}_{p,m-w_i}(f)$, i.e., the search set for the $i$th component is reduced by a factor $p^{-\min\{w_i,m\}}$ in comparison to the standard CBC construction. We will show that under certain conditions on the weights $\bsgamma$ and the parameters $w_i$ a polynomial lattice point set constructed according to our reduced CBC construction has a low weighted star discrepancy of order $N^{-1+\delta}$ for all $\delta>0$. The standard CBC construction (cf. \cite{SR02}) can be done in $\mathcal{O}(sN^2)$ operations. To speed up the construction, in a first step, making use of ideas from Nuyens and Cools~\cite{NC2, NC} on fast Fourier transformation, the construction cost can be reduced to $\mathcal{O}(sN\log{N})$, as for example done in \cite{DLP05}. Combining this with our reduced search sets we obtain a computational cost that is independent of the dimension eventually. Reduced CBC constructions have been introduced first by Dick et al. in~\cite{DKLP15} for lattice and polynomial lattice point sets with a small worst case integration error in Korobov and Walsh spaces, respectively, and have also been investigated in~\cite{KL15} for lattice point sets with small weighted star discrepancy. \vspace{3mm}

An interesting aspect of the discrepancy of high dimensional point sets is the so-called tractability of discrepancy (see, e.g.,~\cite{NW1,NW2,NW3} for detailed information). For $N,s\in\NN$ let
$$\mathrm{disc}_{\infty}(N,s):=\inf_{\substack{\cP\subseteq [0,1)^s \\ \#\cP=N}}D_{N,\bsgamma}^{\ast}(\cP),$$ the $N$th minimal star discrepancy. To introduce the concept of tractability of discrepancy we define the information complexity (also called the inverse of the weighted star discrepancy) as
$$ N^{\ast}(s,\varepsilon):=\min\{N\in\NN \mid \mathrm{disc}_{\infty}(N,s) \leq \varepsilon\}. $$
Thus $N^{\ast}(s,\varepsilon)$ is the minimal number of points required to achieve a weighted star discrepancy of at most $\varepsilon$. To keep the construction cost of our generating vector low, it is, of course, beneficial to have a small information complexity and thus to stand a chance to have a polynomial lattice point set of small size. This is why we are interested in how fast the information complexity grows when $s$ and $\varepsilon^{-1}$ tend to infinity. 
Tractability describes this dependence of the information complexity on the dimension $s$ and the error demand $\varepsilon$. The best we can hope for is the case where $N^{\ast}(s,\varepsilon)$ is independent of $s$ and depends at most polynomially on $\varepsilon^{-1}$. To be more precise, we say that we achieve strong polynomial tractability if there exist constants $C,\tau>0$ such that
$$ N^{\ast}(s,\varepsilon)\leq C\varepsilon^{-\tau} $$
for all $s\in\NN$ and all $\varepsilon\in (0,1)$. Roughly speaking, a problem is considered tractable if its information complexity's dependence on $s$ and $\varepsilon^{-1}$ is not exponential. Taking weights into account in the definition of discrepancy can sometimes overcome the so-called curse of dimensionality, i.e., an exponential dependence of $N^{\ast}(s,\varepsilon)$ on $s$. We will show that our reduced fast CBC algorithm finds a generating vector $\bsg$ of a polynomial lattice point set that achieves strong polynomial tractability provided that
$$ \sum_{j=1}^{\infty}\gamma_j p^{w_j}<\infty $$
with a construction cost of
$$ \mathcal{O}\left(N + \min\{s,t\}N + N\sum_{d=1}^{\min\{s,t\}}(m-w_d)p^{-w_d}\right) $$
operations, where $t=\max\{j\in\NN\mid w_j<m\}$.\vspace{5mm}

\noindent
Before stating our main results we would like to discuss a motivating example. Consider first the standard CBC construction as treated in~\cite{DKLP07,DLP05}, where $w_j=0$ for all $j\geq 0$. In this case, a sufficient condition for strong polynomial tractability is $\sum_{j=1}^{\infty}\gamma_j<\infty$, which for instance is satisfied for the special choices $\gamma_j=j^{-2}$ and $\gamma_j=j^{-1000}$. However, in the second example the weights decay much faster than in the first. We can make use of this fact by introducing the sequence $\bsw=(w_j)_{j\geq 0}$ such that the condition $\sum_{j=1}^{\infty}\gamma_jp^{w_j}<\infty$ holds, while still achieving strong polynomial tractability (see Corollary~\ref{cor:tractabilitykrilaineupoly}). This way, we can reduce the size of the search sets for the components of the generating vector if the weights $\gamma_j$ decay very fast. Consider for example the weight sequence $\gamma_j=j^{-k}$ for some $k>1$. For $w_j=\lfloor (k-\alpha)\log_{p}{j}\rfloor$ with arbitrary $1<\alpha<k$ we find
$$ \sum_{j=1}^{\infty}\gamma_jp^{w_j}\leq \sum_{j=1}^{\infty}j^{-k}j^{k-\alpha}= \sum_{j=1}^{\infty}j^{-\alpha}=\zeta(\alpha)<\infty,$$
where $\zeta$ denotes the Riemann Zeta function. Observe that for large $k$, i.e., fast decaying weights, we may choose smaller search sets and thereby speed up the CBC algorithm.\vspace{5mm}
\noindent
This paper is organized as follows. In the next section we give an algorithm for constructing polynomial lattice point sets and we derive an upper bound on the weighted star discrepancy of the point set constructed with this algorithm. We also give tractability results and analyze the computational cost of our algorithm. At first, we consider the case where $f: \FF_p\to\FF_p, x \mapsto x^m$. Then we consider the case where the modulus of the polynomial lattice point set is irreducible. 

\section{A reduced CBC construction}

In this section we present a CBC construction for the vector $\left( x^{w_1}g_1 , \dotsc , x^{w_s} g_s \right)$ and an upper bound for the weighted star discrepancy of the corresponding polynomial lattice point set.

First note that if $\bsg\in G_{p,m}^s$, then it is known (see~\cite{DLP05}) that
\begin{equation} \label{generalkrilaineupoly} D_{N,\bsgamma}^{\ast}(\bsg,f)\leq \sum_{\substack{\mathfrak{u}\subseteq [s] \\ \mathfrak{u}\neq \emptyset}}\gamma_{\mathfrak{u}}\left(1-\left(1-\frac{1}{N}\right)^{|\mathfrak{u}|}\right)+R^s_{\bsgamma}(\bsg,f), \end{equation}
where in the case of product weights we have
\begin{equation} \label{rgammakrilaineupoly} R^s_{\bsgamma}(\bsg,f)=\sum_{\substack{\bsh\in G_{p,m}^s\setminus \{\bszero\} \\ \bsh \cdot \bsg \equiv 0 \bmod{f}}}\prod_{i=1}^{s}r_p(h_i,\gamma_i). \end{equation}
Here, for elements $\bsh=(h_1,\dots,h_s)$ and $\bsg=(g_1,\dots,g_s)$ in $G_{p,m}^s$ we define the scalar product by $\bsh \cdot \bsg:=h_1g_1+\dots+h_sg_s$. The numbers $r_p(h,\gamma)$ for $h\in G_{p,m}$ and $\gamma \in \RR$ are defined as
$$  r_p(h,\gamma)=\begin{cases}
                       1+\gamma & \mbox{if \,} h=0, \\
                       \gamma r_p(h) & \mbox{otherwise,}
                      \end{cases}  
$$
where for $h=h_0+h_1x+\dots+h_{a}x^{a}$ with $h_a\neq 0$ we set
$$ r_p(h)=\frac{1}{p^{a+1}\sin^2\left(\frac{\pi}{p}h_{a}\right)}.$$
Thus, in order to analyze the weighted star discrepancy of a polynomial lattice point set it suffices to investigate the quantity $R^s_{\bsgamma}(\bsg,f)$. This is due to the result of Joe \cite{J06}, who proved that for any summable weight sequence $(\gamma_j)_{j \geq 1}$ we have
\begin{align*}
 \sum_{\substack{\mathfrak{u}\subseteq [s] \\ \mathfrak{u}\neq \emptyset}}\gamma_{\mathfrak{u}}\left(1-\left(1-\frac{1}{N}\right)^{|\mathfrak{u}|}\right) \leq \frac{\max(1,\Gamma)\mathrm{e}^{\sum_{i=1}^{\infty}\gamma_i}}{N}\,,
\end{align*}
with $\Gamma:=\sum_{i=1}^{\infty}\frac{\gamma_i}{1+\gamma_i}$.

\begin{algorithm}\label{reducedCbckrilaineupoly}
Let $p\in\PP$, $m\in\NN$, $f\in \FF_p[x]$ and let $(w_j)_{j \geq 1}$ be a non-decreasing sequence of nonnegative integers and consider product weights $(\gamma_j)_{j \geq 1}$. Construct
$\left( g_1 , \dotsc , g_s \right) \in \mathcal{G}_{p,m-\bsw}^s(f)$ as follows:
\begin{enumerate}
	\item Set $g_1 = 1$.
	\item For $d \in [s-1]$ assume $(g_1, \dotsc , g_d)\in \mathcal{G}_{p,m-\bsw}^d(f)$ to be already found. Choose $g_{d+1} \in \mathcal{G}_{p,m-w_{d+1}}(f)$ such that
	$$R_{\bsgamma}^{d+1}{(x^{w_1}g_1, \dotsc , x^{w_d}g_d, x^{w_{d+1}}g_{d+1})}$$
	is minimized as a function of $g_{d+1}$.
	\item Increase $d$ by 1 and repeat the second step until $\left(g_1 , \dotsc , g_s \right)$ is found.
\end{enumerate}
\end{algorithm}

\begin{remark}
Of course we have $\mathcal{G}_{p,m-\bsw}^s(f) \subseteq G_{p,m}^s$, and thus in Algorithm~\ref{reducedCbckrilaineupoly} it indeed suffices to consider $R_{\bsgamma}^{d+1}$ rather than the weighted star discrepancy. 
\end{remark}

In the algorithm above, the search set is reduced for each coordinate of $\left( g_1 , \dotsc , g_s \right)$ according to its importance, as with increasing $w_j$ the search set becomes smaller, as the weight $\gamma_j$ and thus the corresponding component's influence on the quality of the generating vector decreases. For this reason we call Algorithm~\ref{reducedCbckrilaineupoly} a reduced CBC algorithm.  We will now study Algorithm \ref{reducedCbckrilaineupoly} for different choices of $f$.

\subsection{Polynomial lattice point sets for $f(x)=x^m$}
We will now study the interesting case where $f \colon \FF_p \to \FF_p, x \mapsto x^m$. Throughout the rest of this section we write $x^m$ instead of $f$ to emphasize our special choice of $f$. Note that for $g \in \FF_p((x^{-1}))$ the Laurent series $g/f$ can be easily computed in this case by shifting the coefficients of $g$ $m$ times to the left. It is the aim of this section to prove the following theorem:

\begin{theorem} \label{theokrilaineupoly}
   Let $\bsgamma = (\gamma_j)_{j \geq 1}$ and $\bsw$ with $0 = w_1 \leq w_2 \leq \cdots$. Let further $(g_1,\dots,g_s)\in \mathcal{G}_{p,m-\bsw}^s(x^m)$ be constructed using Algorithm \ref{reducedCbckrilaineupoly}. Then we have for every $d\in [s]$
     $$ R_{\bsgamma}^d((x^{w_1}g_1,\dots,x^{w_d}g_d),x^m)\leq \frac{1}{p^m}\prod_{i=1}^d \left(1+\gamma_i+\gamma_i 2p^{\min\{w_i,m\}}m\frac{p^2-1}{3p}\right). $$
\end{theorem}
As a direct consequence we obtain  the following discrepancy estimate.
\begin{corollary} \label{coro2krilaineupoly}
Let $N=p^m$ and $\bsgamma$, $\bsw$ and $(g_1,\dots,g_s)$ as in Theorem~\ref{theokrilaineupoly}. Then the polynomial lattice point set $\cP\left( (x^{w_1}g_1 , \dotsc , x^{w_s} g_s),x^m \right)$ has a weighted star discrepancy
\begin{align}\label{eq:boundDiscrkrilaineupoly}
D_{N,\bsgamma}^*&\left( (x^{w_1}g_1 , \dotsc , x^{w_s} g_s),x^m \right) \nonumber\\
 &\leq \sum_{\substack{\mathfrak{u}\subseteq [s] \\ \mathfrak{u}\neq \emptyset}}{\gamma_{\mathfrak{u}}\left( 1- \left( 1 - \frac{1}{N} \right)^{|\mathfrak{u}|} \right)} + \frac{1}{N} \prod_{i=1}^{s}{\left( 1+\gamma_i + \gamma_i 2 p^{\min{\{ w_i, m \}}} m\frac{p^2-1}{3p}  \right)}.
\end{align}
\end{corollary}


Knowing the above discrepancy bound, we are now ready to ask about the size of the polynomial lattice point set required to achieve a weighted star discrepancy not exceeding some $\varepsilon$ threshold. In particular, we would like to know how this size depends on the dimension $s$ and on $\varepsilon$.

\begin{corollary}\label{cor:tractabilitykrilaineupoly}
Let $N=p^m$, $\bsgamma$ and $\bsw$ as in Theorem~\ref{theokrilaineupoly} and consider the problem of constructing generating vectors for polynomial lattice point sets with small weighted star discrepancy. Then
\begin{align*}
\sum_{j=1}^\infty \gamma_j p^{w_j} < \infty
\end{align*}
is a sufficient condition for strong polynomial tractability. This condition further implies $D_{N,\bsgamma}^*\left( (x^{w_1}g_1 , \dotsc , x^{w_s} g_s),x^m \right) = \mathcal{O}(N^{-1 +\delta})$, with the implied constant independent of $s$, for any $\delta >0$, where $(g_1,\dots,g_s)\in \mathcal{G}_{p,m-\bsw}^s(x^m)$ is constructed using Algorithm \ref{reducedCbckrilaineupoly}.
\end{corollary} 

\begin{proof}
Construct a generating vector $(g_1, \dots, g_s) \in \mathcal{G}_{p,m-w}^s(x^m)$ applying Algorithm~\ref{reducedCbckrilaineupoly} and consider its weighted star discrepancy, bounded by \eqref{eq:boundDiscrkrilaineupoly}. Following closely the lines of the argumentation in \cite[Section 5]{KL15} and noticing that $2m \frac{p^2-1}{3p} = \mathcal{O}(\log N)$ we obtain the result. More precisely, provided that the $\gamma_j p^{w_j}$'s are summable, we have a means to construct polynomial lattice point sets $\mathcal{P}(\bsg,f)$ with $D^*_{N,\bsgamma}(\bsg,f)\leq \varepsilon$, whose sizes grow polynomially in $\varepsilon^{-1}$ and are independent of the dimension. As a result the problem is strongly polynomially tractable. The discrepancy result $D_{N,\bsgamma}^*\left( (x^{w_1}g_1 , \dotsc , x^{w_s} g_s),x^m \right) = \mathcal{O}(N^{-1 +\delta})$ also follows directly from \cite{KL15}. 
\end{proof}

\begin{remark}
Recall that $t = \max\{ j \in \NN \colon w_j < m \}$ and note that setting $w_j = m$ for all $j > t$ does neither change the bound on the weighted star discrepancy nor the computational cost of Algorithm \ref{reducedCbckrilaineupoly}. It might change the generating vector though. If so, however, only components with very little influence on the quality of the point set are altered. Defining $w_j = m$ for all $j > t$, it suffices to have a summable weight sequence $\bsgamma$ in order to achieve strong polynomial tractability, as long as $t$ is finite.
\end{remark} 

In order to show Theorem~\ref{theokrilaineupoly} we need several auxiliary results. 

\begin{lemma} \label{summekrilaineupoly} Let $a\in \FF_p[x]$ be monic. Then we have 
   $$\sum_{\substack{h\in G_{p,m}\setminus\{0\}\\ a\mid h}}r_p(h)=\left(m-\deg(a)\right)\frac{p^2-1}{3p}p^{-\deg(a)}.$$
	In particular, for $a=1$ this formula yields
	 $$ \sum_{h\in G_{p,m}\setminus\{0\}}r_p(h)=m\frac{p^2-1}{3p}. $$
\end{lemma}

\begin{proof} This fact follows from \cite[p. 1055]{DKLP07} (by setting $\gamma_{d+1}=1$). The special case $a=1$ also follows from \cite[Lemma 2.2]{DLP05} by setting $s=1$.
\end{proof}

For our purposes, it is convenient to write $R^s_{\bsgamma}(\bsg,f)$ from~\eqref{rgammakrilaineupoly} in an alternative way. To this end, we introduce some notation. For a Laurent series $L\in \mathbb{F}_p((x^{-1}))$ we denote by $c_{-1}(L)$ its coefficient of $x^{-1}$, i.e., its residuum.
Further, we set $X_p(L):=\chi_p(c_{-1}(L))$, where $\chi_p$ is a non-trivial additive character of $\mathbb{F}_p$. One could for instance choose
$\chi_p(n)=\mathrm{e}^{\frac{2\pi\ii}{p} n}$ for $n\in \FF_p$ (see, e.g.,~\cite{Niedfin}). It is clear that $X_p(L)=1$ if $L$ is a polynomial and that $X_p(L_1+L_2)=X_p(L_1)X_p(L_2)$
for $L_1,L_2\in\mathbb{F}_p((x^{-1}))$. From \cite[p. 78]{Nied92book} we know that
\begin{equation} \label{Niedformelkrilaineupoly} \sum_{v\in G_{p,m}}X_p\left(\frac{v}{f}g\right)=\begin{cases}
                                           p^m & \mbox{if \,} f\mid g, \\
                                           0 & \mbox{otherwise.}
                                       \end{cases} \end{equation}

\begin{lemma} \label{alternativekrilaineupoly} We have $$ 
   R^s_{\bsgamma}(\bsg,f)=-\prod_{i=1}^s (1+\gamma_i)+\frac{1}{p^m}\sum_{v\in G_{p,m}}\prod_{i=1}^{s}\left(1+\gamma_i+\gamma_i\sum_{h \in G_{p,m}\setminus \{0\}}r_p(h)X_p\left(\frac{v}{f}h g_i\right)\right).  $$
\end{lemma}

\begin{proof}
We employ the properties of $X_p$ as stated above to obtain from~\eqref{rgammakrilaineupoly}
\begin{align*}
   R^s_{\bsgamma}(\bsg,f)=&-\prod_{i=1}^s (1+\gamma_i)+\frac{1}{p^m}\sum_{\bsh\in G_{p,m}^s}\left(\prod_{i=1}^{s}r_p(h_i,\gamma_i)\right)\sum_{v\in G_{p,m}}X_p\left(\frac{v}{f}\bsh\cdot \bsg\right) \\
   =&-\prod_{i=1}^s (1+\gamma_i)+\frac{1}{p^m}\sum_{v\in G_{p,m}}\prod_{i=1}^{s}\left(\sum_{h_i \in G_{p,m}}r_p(h_i,\gamma_i)X_p\left(\frac{v}{f}h_ig_i\right)\right) \\
   =&-\prod_{i=1}^s (1+\gamma_i)+\frac{1}{p^m}\sum_{v\in G_{p,m}}\prod_{i=1}^{s}\left(1+\gamma_i+\gamma_i\sum_{h \in G_{p,m}\setminus \{0\}}r_p(h)X_p\left(\frac{v}{f}h g_i\right)\right),
\end{align*}
and the claimed formula is verified.
\end{proof}

Now we study a sum which will appear later in the proof of Theorem~\ref{theokrilaineupoly} and show an upper bound for it.

\begin{lemma} \label{lemma2krilaineupoly} Let $w\in\NN_0$ and $v\in G_{p,m}$. Let 
$$Y_{p^m,w}(v,x^m):=\sum_{g\in \mathcal{G}_{p,m-w}(x^m)}\sum_{h\in G_{p,m}\setminus\{0\}}r_p(h)X_p\left(\frac{v}{x^m}h x^{w}g\right),$$
where $x^w$ denotes the polynomial $f(x)=x^w$. Then we have
$$\frac{1}{\#\mathcal{G}_{p,m-w}(x^m)}\sum_{v\in G_{p,m}}|Y_{p^m,w}(v,x^m)|\leq 2p^{\min\{w,m\}}m\frac{p^2-1}{3p}.$$
\end{lemma}

\begin{proof}
   Let us first assume that $w\geq m$. Then we have $\mathcal{G}_{p,m-w}(x^m)=\{1\}$ and therefore
   \begin{align*}
       Y_{p^m,w}(v,x^m)=\sum_{h\in G_{p,m}\setminus\{0\}}r_p(h)X_p(vhx^{w-m})=\sum_{h\in G_{p,m}\setminus\{0\}}r_p(h)=m\frac{p^2-1}{3p}
   \end{align*}
	with Lemma~\ref{summekrilaineupoly}.
   This leads to
   $$\frac{1}{\#\mathcal{G}_{p,m-w}(x^m)}\sum_{v\in G_{p,m}}|Y_{p^m,w}(v,x^m)|=p^m m\frac{p^2-1}{3p}\leq 2p^{\min\{w,m\}}m\frac{p^2-1}{3p}$$
   in this case. For the rest of the proof let $w<m$ and additionally we abbreviate $\#\mathcal{G}_{p,m-w}(x^m)$ by $\#\mathcal{G}$. We write
   \begin{align*}
      \frac{1}{\#\mathcal{G}}\sum_{v\in G_{p,m}}|Y_{p^m,w}(v,x^m)|=& \frac{1}{\#\mathcal{G}}\sum_{\substack{v\in G_{p,m} \\ x^{m-w}\mid v}}|Y_{p^m,w}(v,x^m)|+
      \frac{1}{\#\mathcal{G}}\sum_{\substack{v\in G_{p,m} \\ x^{m-w} \nmid v}}|Y_{p^m,w}(v,x^m)|.
   \end{align*}
   In what follows, we refer to the latter sums as
   $$ S_1:=\frac{1}{\#\mathcal{G}}\sum_{\substack{v\in G_{p,m} \\ x^{m-w}\mid v}}|Y_{p^m,w}(v,x^m)| \text{\quad and \, \,}
     S_2:=\frac{1}{\#\mathcal{G}}\sum_{\substack{v\in G_{p,m} \\ x^{m-w} \nmid v}}|Y_{p^m,w}(v,x^m)|. $$
  We may uniquely write any $v\in G_{p,m}\setminus\{0\}$ in the form $v=qx^{m-w}+\ell,$ where $q,\ell\in\FF_q[x]$ with $\deg(q)<w$ and $\mathrm{deg}(\ell)<m-w$. Using the properties of $X_p$ it is clear that $Y_{p^m,w}(v,x^m)=Y_{p^m,w}(\ell,x^m)$ and hence
  \begin{align*}
     S_1=&\frac{1}{\#\mathcal{G}}\sum_{\substack{v\in G_{p,m} \\ x^{m-w}\mid v}}|Y_{p^m,w}(0,x^m)|
       =\sum_{\substack{v\in G_{p,m} \\ x^{m-w}\mid v}}\frac{1}{\#\mathcal{G}}\sum_{g\in \mathcal{G}_{p,m-w}(x^m)}\sum_{h\in G_{p,m}\setminus\{0\}}r_p(h)  \\
       =&\sum_{\substack{v\in G_{p,m} \\ x^{m-w}\mid v}}m\frac{p^2-1}{3p}= p^{\min\{w,m\}} m\frac{p^2-1}{3p}.
  \end{align*}
  
  We move on to $S_2$. Let $e(\ell):=\max\{k\in\{0,1,\dots,m-w-1\}: x^k\mid \ell\}$. With this definition we may display $S_2$ as
	\begin{equation} \label{S2krilaineupoly} S_2=\frac{p^w}{\#\mathcal{G}}\sum_{k=0}^{m-w-1}\sum_{\substack{\ell\in G_{p,m-w}\setminus \{0\} \\ e(\ell)=k}}|Y_{p^m,w}(\ell,x^m)|. \end{equation}
	In the following, we compute $Y_{p^m,w}(\ell,x^m)$ for $\ell\in G_{p,m-w}\setminus \{0\}$ with $e(\ell)=k$. Let $\mu_p$ be the M\"{o}bius function on the set of monic polynomials over $\FF_p$, i.e., $\mu_p:\FF_p[x] \rightarrow \{-1,0,1\}$ and $$\mu_p(h)=\begin{cases}(-1)^{\nu} & \text{if $h$ is squarefree and has $\nu$ irreducible factors,} \\ 0 & \text{else.} \end{cases}$$ 
	The fact that $\mu_p(1)=1$, $\mu_p(x)=-1$ and $\mu_p(x^i)=0$ for $i\in\NN$, $i\geq 2$, yields the equivalence of $\sum_{t\mid \gcd(x^{m-w},g)}\mu_p(t)=1$ and $\gcd(x^{m-w},g)=1$. Therefore we can write
	\begin{align*}
	   Y_{p^m,w}(\ell,x^m)=&\sum_{h\in G_{p,m}\setminus\{0\}}r_p(h)\sum_{g\in G_{p,m-w}}X_p\left(\frac{\ell}{x^{m-w}}hg\right)\sum_{t\mid \gcd(x^{m-w},g)}\mu_p(t) \\
		=&\sum_{h\in G_{p,m}\setminus\{0\}}r_p(h)\sum_{t\mid x^{m-w}}\mu_p(t)\sum_{\substack{g\in G_{p,m-w} \\ t\mid g}}X_p\left(\frac{\ell}{x^{m-w}}hg\right) \\
		=&\sum_{h\in G_{p,m}\setminus\{0\}}r_p(h)\sum_{t\mid x^{m-w}}\mu_p(t)\sum_{\substack{a\in G_{p,m-w-\deg(t)}}}X_p\left(\frac{\ell}{x^{m-w}}hat\right) \\
		=&\sum_{h\in G_{p,m}\setminus\{0\}}r_p(h)\sum_{t\mid x^{m-w}}\mu_p\left(\frac{x^{m-w}}{t}\right)\sum_{\substack{a\in G_{p,\deg(t)}}}X_p\left(\frac{a}{t}h\ell\right) \\
		=&\sum_{h\in G_{p,m}\setminus\{0\}}r_p(h)\sum_{\substack{t\mid x^{m-w}\\ t \mid h\ell}}\mu_p\left(\frac{x^{m-w}}{t}\right)p^{\deg(t)} \\
		=&\sum_{t\mid x^{m-w}}\mu_p\left(\frac{x^{m-w}}{t}\right)p^{\deg(t)}\sum_{\substack{h\in G_{p,m}\setminus\{0\} \\ t \mid h\ell}}r_p(h).
	\end{align*}
	The equivalence of the conditions $t \mid h\ell$ and $\frac{t}{\gcd(t,\ell)}\mid h$ yields
	$$ Y_{p^m,w}(\ell,x^m)=\sum_{t\mid x^{m-w}}\mu_p\left(\frac{x^{m-w}}{t}\right)p^{\deg(t)}\sum_{\substack{h\in G_{p,m}\setminus\{0\} \\ \frac{t}{\gcd(t,\ell)}\mid h}}r_p(h). $$
	We investigate the inner sum and use Lemma~\ref{summekrilaineupoly} with $a=\frac{t}{\gcd(t,\ell)}$ to find
	\begin{align*}
	   \sum_{\substack{h\in G_{p,m}\setminus\{0\} \\ \frac{t}{\gcd(t,\ell)}\mid h}}r_p(h)=
		\left(m-\deg\left(\frac{t}{\gcd(t,\ell)}\right)\right)\frac{p^2-1}{3p}p^{-\deg\left(\frac{t}{\gcd(t,\ell)}\right)}.
	\end{align*}
	Now we have
\begin{align*}
  Y_{p^m,w}(\ell,x^m)=&\frac{p^2-1}{3p}\sum_{t\mid x^{m-w}}\mu_p\left(\frac{x^{m-w}}{t}\right)   \left(m-\deg\left(\frac{t}{\gcd(t,\ell)}\right)\right)p^{\deg(\gcd(t,\ell))} \\
	=& \frac{p^2-1}{3p}m\sum_{t\mid x^{m-w}}\mu_p\left(\frac{x^{m-w}}{t}\right)p^{\deg(\gcd(t,\ell))} \\
	&-\frac{p^2-1}{3p}\sum_{t\mid x^{m-w}}\mu_p\left(\frac{x^{m-w}}{t}\right)\deg\left(\frac{t}{\gcd(t,\ell)}\right)p^{\deg(\gcd(t,\ell))}.
\end{align*}
From the fact that $e(\ell)=k\leq m-w-1$ we obtain $\gcd(x^{m-w},\ell)=\gcd(x^{m-w-1},\ell)=x^k$. This observation leads to
\begin{align*}
   \sum_{t\mid x^{m-w}}\mu_p\left(\frac{x^{m-w}}{t}\right)p^{\deg(\gcd(t,\ell))}=p^{\deg(\gcd(x^{m-w},\ell))}-p^{\deg(\gcd(x^{m-w-1},\ell))}=0
\end{align*}
and
\begin{align*}
   \sum_{t\mid x^{m-w}}&\mu_p\left(\frac{x^{m-w}}{t}\right)\deg\left(\frac{t}{\gcd(t,\ell)}\right)p^{\deg(\gcd(t,\ell))}\\
	   =&
	    \deg\left(\frac{x^{m-w}}{\gcd(x^{m-w},\ell)}\right)p^{\deg(\gcd(x^{m-w},\ell))}-\deg\left(\frac{x^{m-w-1}}{\gcd(x^{m-w-1},\ell)}\right)p^{\deg(\gcd(x^{m-w-1},\ell))} \\
		=&(m-w-k)p^k-(m-w-k-1)p^k=p^k.
\end{align*}
Altogether we have
$$  Y_{p^m,w}(\ell,x^m)=-\frac{p^2-1}{3p}p^k. $$
Inserting this result into~\eqref{S2krilaineupoly} yields
\begin{align*}
  S_2=\frac{p^w}{\#\mathcal{G}}\frac{p^2-1}{3p}\sum_{k=0}^{m-w-1}p^k\sum_{\substack{\ell\in G_{p,m-w}\setminus \{0\} \\ e(\ell)=k}}1.
\end{align*}
Since
\begin{align*}
 \#\{\ell &\in G_{p,m-w}\setminus \{0\}: e(\ell)=k\} \\
	 =&\#\{\ell\in G_{p,m-w}\setminus \{0\}:x^k \mid \ell\}-\#\{ \ell\in G_{p,m-w}\setminus \{0\}:x^{k+1} \mid \ell\} \\
	 =& p^{m-w-k}-1-(p^{m-w-k-1}-1)=p^{m-w-k-1}(p-1),
\end{align*}
we have
\begin{align*}
  S_2=&\frac{p^w}{p^{m-w-1}(p-1)}\frac{p^2-1}{3p}\sum_{k=0}^{m-w-1}p^kp^{m-w-k-1}(p-1) \\
	   =&p^w\frac{p^2-1}{3p}(m-w)\leq p^{\min\{w,m\}} m\frac{p^2-1}{3p}.
\end{align*}
Summarizing, we have shown
\begin{align*}
      \frac{1}{\#\mathcal{G}}\sum_{v\in G_{p,m}}|Y_{p^m,w}(v,x^m)|=S_1+S_2\leq 2p^{\min\{w,m\}} m\frac{p^2-1}{3p},
   \end{align*}
	which completes the proof.

\end{proof}

Now we are ready to prove Theorem~\ref{theokrilaineupoly} using induction on $d$.

\begin{proof}
 We show the result for $d=1$. From Lemma~\ref{alternativekrilaineupoly} we have
  \begin{align*}
   R_{\bsgamma}^1((x^{w_1}),x^m)
   =&-(1+\gamma_1)+\frac{1}{p^m}\sum_{v\in G_{p,m}}\left(1+\gamma_1+\gamma_1\sum_{h \in G_{p,m}\setminus \{0\}}r_p(h)X_p\left(\frac{v}{x^m}h x^{w_1}\right)\right) \\ =&\frac{\gamma_1}{p^m}\sum_{v\in G_{p,m}}\sum_{h \in G_{p,m}\setminus \{0\}}r_p(h)X_p\left(\frac{v}{x^m}h x^{w_1}\right).
\end{align*} 
If $w_1\geq m$, then
\begin{align*}
 R_{\bsgamma}^1((x^{w_1}),x^m)
   =&\frac{\gamma_1}{p^m}\sum_{v\in G_{p,m}}\sum_{h \in G_{p,m}\setminus \{0\}}r_p(h)
   =\frac{\gamma_1}{p^m}p^{\min\{w_1,m\}}m\frac{p^2-1}{3p} \\
   &\leq \frac{1}{p^m}\left(1+\gamma_1+\gamma_12p^{\min\{w_1,m\}}m\frac{p^2-1}{3p} \right).
\end{align*} 
If $w_1<m$, then we can write

\begin{align*}
R_{\bsgamma}^1((x^{w_1}),x^m)
   =&\frac{\gamma_1}{p^m}\sum_{v\in G_{p,m}}\sum_{h \in G_{p,m}\setminus \{0\}}r_p(h)X_p\left(\frac{v}{x^m}h x^{w_1}\right) \\
   =&\frac{\gamma_1}{p^m}\sum_{\substack{h \in G_{p,m}\setminus \{0\} \\ x^{m-w_1} \mid h}}r_p(h)\sum_{v\in G_{p,m}}X_p\left(\frac{v}{x^m}h x^{w_1}\right) \\
   &+\frac{\gamma_1}{p^m}\sum_{\substack{h \in G_{p,m}\setminus \{0\} \\ x^{m-w_1} \nmid h}}r_p(h)\sum_{v\in G_{p,m}}X_p\left(\frac{v}{x^m}h x^{w_1}\right) \\
   =&\gamma_1 \sum_{\substack{h \in G_{p,m}\setminus \{0\} \\ x^{m-w_1} \mid h}}r_p(h),
\end{align*}
where we used~\eqref{Niedformelkrilaineupoly} in the latter step.
We regard Lemma~\ref{summekrilaineupoly} with $a=x^{m-w_1}$ to compute
$$\sum_{\substack{h \in G_{p,m}\setminus \{0\} \\ x^{m-w_1} \mid h}}r_p(h)=\frac{1}{p^m}p^{w_1}w_1\frac{p^2-1}{3p}\leq \frac{1}{p^m}p^{\min\{w_1,m\}}m\frac{p^2-1}{3p},  $$
which leads to the desired result also in this case.\\ 
Now let $d\in [s-1]$. Assume that we have some $(g_1,\dots,g_d) \in \mathcal{G}_{p,m-\bsw}^d(x^m)$ such that
$$ R_{\bsgamma}^d((x^{w_1}g_1,\dots,x^{w_d}g_d),x^m)\leq \frac{1}{p^m}\prod_{i=1}^d \left(1+\gamma_i+\gamma_i 2p^{\min\{w_i,m\}}m\frac{p^2-1}{3p}\right).$$ 

Let $g^{\ast}\in \mathcal{G}_{p,m-w_{d+1}}(x^m)$ be such that $R_{\bsgamma}^{d+1}((x^{w_1}g_1,\dots,x^{w_d}g_d,x^{w_{d+1}}g_{d+1}),x^m)$ is minimized as a function of $g_{d+1}$ for $g_{d+1}=g^{\ast}$. Then we have
\begin{align}\label{Lrealkrilaineupoly}
   R_{\bsgamma}^{d+1}&((x^{w_1}g_1,\dots,x^{w_d}g_d,x^{w_{d+1}}g^{\ast}),x^m)=-(1+\gamma_{d+1})\prod_{i=1}^d(1+\gamma_i) \nonumber\\
       +& \frac{1}{p^m}\sum_{v\in G_{p,m}} \prod_{i=1}^d \left(1+\gamma_i+\gamma_i\sum_{h\in G_{p,m}\setminus\{0\}}r_p(h)X_p\left(\frac{v}{x^m}hx^{w_i}g_i\right)\right) \nonumber\\
        &\times  \left(1+\gamma_{d+1}+\gamma_{d+1}\sum_{h\in G_{p,m}\setminus\{0\}}r_p(h)X_p\left(\frac{v}{x^m}h x^{w_{d+1}}g^{\ast}\right)\right) \nonumber\\
     =& (1+\gamma_{d+1}) R_{\bsgamma}^{d}((x^{w_1}g_1,\dots,x^{w_d}g_d),x^m)+L(g^{\ast}),
\end{align}
where   
   \begin{align*}
   L(g^{\ast})=&\frac{\gamma_{d+1}}{p^m}\sum_{v\in G_{p,m}}  \sum_{h\in G_{p,m}\setminus\{0\}}r_p(h)X_p\left(\frac{v}{x^m}h x^{w_{d+1}}g^{\ast}\right) \\ &\times  
   \prod_{i=1}^d \left(1+\gamma_i+\gamma_i\sum_{h\in G_{p,m}\setminus\{0\}}r_p(h)X_p\left(\frac{v}{x^m}h x^{w_{i}}g_i\right)\right).
   \end{align*} 
A minimizer $g^{\ast}$ of $R_{\bsgamma}^{d+1}((x^{w_1}g_1,\dots,x^{w_d}g_d,x^{w_{d+1}}g_{d+1}),x^m)$ is also a minimizer of $L(g_{d+1})$. Combining \eqref{rgammakrilaineupoly} and \eqref{Lrealkrilaineupoly} we obtain that $R_{\bsgamma}^d(\bsg,f)\in \RR$ for all $d \in [s]$. Moreover with equation \eqref{thetakrilaineupoly}, established later on in Section \ref{sectirredkrilaineupoly},  and the fact that $r_p(h,\gamma)>0$ for all $h\in G_{p,m}$ and $\gamma \in (0,1]$, we get that $L(g) \in \RR^+$ for all $g \in \mathcal{G}_{p,m-w_{d+1}}(x^m)$.  Thus we may bound $L(g^{\ast})$ by the mean over all $g\in \mathcal{G}_{p,m-w_{d+1}}(x^m)$, hence
\begin{align*}
    L(g^{\ast})\leq &\frac{1}{\#\mathcal{G}_{p,m-w_{d+1}}(x^m)}\sum_{g_{d+1} \in \mathcal{G}_{p,m-w_{d+1}}(x^m)}L(g_{d+1}) \\
         \leq & \frac{\gamma_{d+1}}{p^m}\sum_{v\in G_{p,m}}\frac{1}{\#\mathcal{G}_{p,m-w_{d+1}}(x^m)} \\ &\times\left|\sum_{g_{d+1}\in \mathcal{G}_{p,m-w_{d+1}}(x^m)}\sum_{h\in G_{p,m}\setminus\{0\}}r_p(h)X_p\left(\frac{v}{x^m}h x^{w_{d+1}}g_{d+1}\right)\right| \\
         &\times  \prod_{i=1}^d \left(1+\gamma_i+\gamma_i\sum_{h\in G_{p,m}\setminus\{0\}}r_p(h)\left|X_p\left(\frac{v}{x^m}h x^{w_{i}}g_i\right)\right|\right) \\
         \leq& \frac{\gamma_{d+1}}{p^m}\prod_{i=1}^d \left(1+\gamma_i+\gamma_i m\frac{p^2-1}{3p}\right)
\sum_{v\in G_{p,m}}\frac{|Y_{p^m,w_{d+1}}(v,x^m)|}{\#\mathcal{G}_{p,m-w_{d+1}}(x^m)},\end{align*}
where we used the estimate $\left|X_p\left(\frac{v}{x^m}h x^{w_{i}}g_i\right)\right|\leq 1$ in the last step. With the induction hypothesis and Lemma~\ref{lemma2krilaineupoly} this leads to
\begin{align*}
     R_{\bsgamma}^{d+1}&((x^{w_1}g_1,\dots,x^{w_d}g_d,x^{w_{d+1}}g^{\ast}),x^m)  \\ \leq& (1+\gamma_{d+1}) \frac{1}{p^m}\prod_{i=1}^d \left(1+\gamma_i+\gamma_i 2p^{\min\{w_i,m\}}m\frac{p^2-1}{3p}\right) \\
      &+\frac{\gamma_{d+1}}{p^m}\prod_{i=1}^d \left(1+\gamma_i+\gamma_i m\frac{p^2-1}{3p}\right)
2p^{\min\{w_{d+1},m\}}m\frac{p^2-1}{3p} \\ 
 \leq& \frac{1}{p^m}\prod_{i=1}^d \left(1+\gamma_i+\gamma_i 2p^{\min\{w_i,m\}}m\frac{p^2-1}{3p}\right)\left(1+\gamma_{d+1}+\gamma_{d+1} 2p^{\min\{w_{d+1},m\}}m\frac{p^2-1}{3p}\right) \\
    =&\frac{1}{p^m}\prod_{i=1}^{d+1} \left(1+\gamma_i+\gamma_i 2p^{\min\{w_i,m\}}m\frac{p^2-1}{3p}\right).
\end{align*}
\end{proof}

\noindent
{\bfseries The reduced fast CBC construction}\newline 
So far we have seen how to construct a generating vector $\bsg$ of the point set $\mathcal{P}(\bsg,x^m)$. In fact Algorithm \ref{reducedCbckrilaineupoly} can be made much faster using results of \cite{DKLP15,NC, NC2}. In this section we are investigating and improving Algorithm \ref{reducedCbckrilaineupoly} and additionally analyzing the computational cost of the improved algorithm.\\
As explained in the following lines Walsh functions are a suitable tool for analyzing the computational cost of CBC algorithms for constructing polynomial lattice point sets. Let $\omega=\mathrm{e}^{2\pi i/p}$, $x\in[0,1)$ and $h$ a nonnegative integer with base $p$ representation $x=x_1/p + x_2/p^2+\ldots$ and $h=h_0 +h_1p +\ldots +h_rp^r$, respectively. Then we define $$\mathrm{wal}_h:[0,1)\rightarrow \mathbb{C}, \mathrm{wal}_h(x):=\omega^{h_0x_1+\ldots+h_rx_{r+1}}.$$ 
The Walsh function system $\{\wal_h ~|~ h=0,1,\ldots\}$ is a complete orthonormal basis in $L_2([0,1))$ which has been used in the analysis of the discrepancy of digital nets (an important class of low-discrepancy point sets which contains polynomial lattice point sets) several times before, see for example \cite{DLP05,H94,LP03}. For further information on Walsh functions see \cite[Appendix A]{DP}.\\
Let $d\geq 1$, $N=p^m$. For $P(\bsg,f)= \{\bsx_0,\ldots,\bsx_{p^m-1}\}$ with $\bsx_n=(x_n^{(1)}, \ldots, x_n^{(s)})$ we have the formula (see \cite[Section 4]{DLP05})
$$ \frac{1}{p^m} \sum_{n=0}^{p^m-1}\prod_{i=1}^{s}\wal_{h_i}(x^{(i)}_n)=\begin{cases} 1 & \text{if } \bsg\cdot \bsh \equiv 0 \pmod{f}, \\ 0 & \text{otherwise,} \end{cases}$$
which allows us to rewrite $R^d_{\bsgamma}(\bsg,x^m)$ in the following way 
\begin{align*}
R^d_{\bsgamma}(\bsg,x^m) = -\prod_{i=1}^d(1+\gamma_i) + \dfrac{1}{p^m}\sum_{n=0}^{p^m-1}\prod_{i=1}^d\sum_{h=0}^{p^m-1}r_p(h,\gamma_i)\mathrm{wal}_h\left(\phi_m\left(\frac{nx^{w_i}g_i}{x^m}\right)\right).
\end{align*}
Note that $r_p(h,\gamma)$ is defined as in \eqref{rgammakrilaineupoly} and we identify the integer in base $p$ representation $h = h_0 +h_1p +\ldots +h_rp^r$ with the polynomial $h(x)=h_0 +h_1x +\ldots +h_rx^r$.
If we set $\psi(\frac{nx^{w_i}g_i}{x^m}):=\sum_{h=1}^{p^m-1}r_p(h)\mathrm{wal}_h(\phi_m(\frac{nx^{w_i}g_i}{x^m}))$ we get that
\begin{align} 
R^d_{\bsgamma}(\bsg,x^m) &= -\prod_{i=1}^d(1+\gamma_i) + \dfrac{1}{p^m}\sum_{n=0}^{p^m-1}\prod_{i=1}^d\left(1+\gamma_i+\gamma_i\psi\left(\frac{nx^{w_i}g_i}{x^m}\right)\right) \nonumber \\
\label{rgCostkrilaineupoly} &= -\prod_{i=1}^d(1+\gamma_i)+ \dfrac{1}{p^m}\sum_{n=0}^{p^m-1}\eta_d(n), 
\end{align}
where $\eta_d(n)=\prod_{i=1}^d\left(1+\gamma_i+\gamma_i\psi(\frac{nx^{w_i}g_i}{x^m})\right)$.\\
In \cite[Section 4]{DLP05} it is proved that we can compute the at most $N$ different values of $\psi(\frac{r}{x^m})$ for $r\in G_{p,m}$ in $\mathcal{O}(N)$ operations. \\
Let us now analyze one step of the reduced CBC Algorithm \ref{reducedCbckrilaineupoly}. Assuming we already have found $(g_1,\ldots,g_d)\in \mathcal{G}^d_{p,m-\bsw}(x^m)$ we have to minimize 
$$R^{d+1}_{\bsgamma}((x^{w_1}g_1,\ldots,x^{w_{d+1}}g_{d+1}),x^m) $$
as a function of $g_{d+1} \in \mathcal{G}_{p,m-w_{d+1}}(x^m)$. If $w_{d+1} \geq m$ then $g_{d+1}=1$ and we are done. Let now $w_{d+1}<m$. From \eqref{rgCostkrilaineupoly} we have that

\begin{align*}
R^{d+1}_{\bsgamma}((x^{w_1}g_1,\ldots,x^{w_{d+1}}g_{d+1}),x^m) &=-\prod_{i=1}^{d+1}(1+\gamma_i)+ \dfrac{1}{p^m}\sum_{n=0}^{p^m-1}\eta_{d+1}(n)\\
&=-\prod_{i=1}^{d+1}(1+\gamma_i)+ \dfrac{1}{p^m}\sum_{n=0}^{p^m-1}\bigg(1+\gamma_{d+1}+\\
&\gamma_{d+1}\psi\left(\frac{nx^{w_{d+1}}g_{d+1}}{x^m}\right)\eta_{d}(n)\bigg).
\end{align*}
In order to minimize $R^{d+1}_{\bsgamma}((x^{w_1}g_1,\ldots,x^{w_{d+1}}g_{d+1}),x^m)$ it is enough to minimize $T_d(g):=\sum_{n=0}^{p^m-1}\psi(\frac{nx^{w_{d+1}}g}{x^m})\eta_{d}(n)$. As in \cite[Section 4]{DKLP15} we can represent this quantity using some specific $(p^{m-w_{d+1}-1}(p-1)\times N)$-matrix $A$ and exploiting its additional structure. Let therefore 
$$A=\left(\psi\left(\frac{nx^{w_{d+1}}g}{x^m}\right)\right)_{\stackrel{n\in\{0,\ldots,N-1\},}{g\in G_{p,m-w_{d+1}}(x^m) }} \text{ and } \bs{\eta}_d=(\eta_d(0),\ldots,\eta_d(N-1))^{\top} .$$
First of all observe that we get $(T(g))_{g\in G_{p,m-w_{d+1}}(x^m)}=A\bs{\eta}_d$. Secondly the matrix A is a block matrix and can be written in the following form
$$A=\left(\Omega^{(m-w_{d+1})}\hdots \Omega^{(m-w_{d+1})}\right) \text{, where } \Omega^{(l)}=\left(\psi\left(\frac{nx^{w_{d+1}}g}{x^m}\right)\right)_{\stackrel{n\in\{0,\ldots p^{l}-1\}}{g\in G_{p,m-w_{d+1}}}(x^m)}. $$ 
If $\bsx$ is any vector of size $p^m$ then we compute
$$A\bsx = \Omega^{(m-w_{d+1})}\bsx_1+ \ldots + \Omega^{(m-w_{d+1})}\bsx_{b^{w_d}} = \Omega^{(m-w_{d+1})}(\bsx_1 + \ldots + \bsx_{b^{w_d}}).$$
With this representation we can apply the machinery of \cite{NC2, NC} and get that multiplication with $\Omega^{(m-w_{d+1})}$ can be done in $\mathcal{O}((m-w_{d+1})p^{m-w_{d+1}})$ operations. 
Summarizing we have:

\begin{algorithm}\label{reducedfastCbckrilaineupoly}~
	\begin{enumerate}
		\item Compute $\psi(\frac{r}{x^m})$ for $r \in G_{p,m}$.
		\item Set $\eta_1(n) = \psi(\frac{nx^{w_1}g_1}{x^m})$ for $n=0,\ldots,p^m-1$.
		\item Set $g_1=1,~ d=2$ and $t=\max\{j\in [s] ~|~ w_j <m\}$.\\
		While $d\leq \min\{s,t\}$,
		\begin{enumerate}
			\item Partition $\eta_{d-1}$ into $p^{w_d}$ vectors $\eta_{d-1}^{(1)},\ldots,\eta_{d-1}^{(p^{w_d})}$ of length $p^{m-w_d}$ and let $\eta^{\prime}=\sum_{i=1}^{p^{w_d}}\eta_{d-1}^{(i)}$.
			\item Let $T_d(g)=\Omega^{(m-w_d)}\eta^{\prime}$.
			\item Let $g_d= \mathrm{argmin}_gT_d(g)$.
			\item Let $\eta_d(n)=\eta_{d-1}(n)\psi(\frac{nx^{w_d}z_d}{x^m})$
			\item Increase $d$ by 1.
		\end{enumerate}
		\item If $ s\geq t$ then set $g_t=g_{t+1}\ldots=g_s=1$.
	\end{enumerate}
\end{algorithm} 

Similar to \cite{DKLP15} we obtain from the observations in this section the following theorem:
\begin{theorem}
The cost of Algorithm \ref{reducedfastCbckrilaineupoly} is
$$\mathcal{O}\left(p^m + \min\{s,t\}p^m + \sum_{d=1}^{\min\{s,t\}}(m-w_d)p^{m-w_d}\right).$$
\end{theorem}

\subsection{Polynomial lattice point sets for irreducible $f$}\label{sectirredkrilaineupoly}
Finally we want to consider the special case where $f$ is an irreducible polynomial. So, for this section let $f$ be an irreducible polynomial over $\FF_p$ with $\deg(f)=m$. 

\begin{theorem} \label{R bound irreduciblekrilaineupoly}
Let $\bsgamma$ and $\bsw$ as in Theorem~\ref{theokrilaineupoly} and let $f \in \FF_p[x]$ be an irreducible polynomial with $\deg(f)=m$. Let further $(g_1,\ldots ,g_s)\in \mathcal{G}_{p,m-\bsw}^s(f)$ be constructed according to Algorithm~\ref{reducedCbckrilaineupoly}. Then we have for every $d\in[s]$
$$R_{\bsgamma}^d((x^{w_1}g_1,\ldots,x^{w_d}g_d),f) \leq \dfrac{1}{p^m}\prod_{i=1}^d\left( 1+ \gamma_i+\gamma_ip^{\min\{w_{i},m\}}m\dfrac{p+1}{3}\right) .$$
\end{theorem}

\begin{proof}
We will prove this result by induction on $d$. According to Algorithm \ref{reducedCbckrilaineupoly} we know that $g_1=1$ for $d=1$. Therefore $R_{\bsgamma}^1((x^{w_1}g_1),f)=0$ since for all $h\in G_{p,m}$ we have $\deg(h)<m$ and hence the congruence $hx^{w_1}\equiv 0 \pmod{f}$ has no solutions.\\
Let $d\in [s-1]$ and assume that we have already found $(g_1,\ldots,g_d) \in \mathcal{G}_{p,m-\bsw}^d(f)$. For $\bsg=(x^{w_1}g_1,\ldots,x^{w_d}g_d)$ we have from~\eqref{rgammakrilaineupoly} that
\begin{align}\label{thetakrilaineupoly}
	R^{d+1}_{\bsgamma}((\bsg,x^{w_{d+1}}g_{d+1}),f)=(1+\gamma_{d+1})R^{d}_{\bsgamma}(\bsg,f) + \theta(g_{d+1}),
\end{align}
where we proceeded similarly as in the proof of Theorem~\ref{theokrilaineupoly}. Here we have $$\theta(g_{d+1})=\sum_{h_{d+1} \in G_{p,m}\back\{0\}}r_p(h_{d+1},\gamma_{d+1})\sum_{\stackrel{\bsh \in G_{p,m}^d}{\bsh \cdot \bsg \equiv -h_{d+1}x^{w_{d+1}}g_{d+1} \pmod{f}}}\prod_{i=0}^dr_p(h_i,\gamma_i).$$
Let $g^*\in \mathcal{G}_{p,m-w_{d+1}}(f)$ be a minimizer of $R^{d+1}_{\bsgamma}((\bsg,x^{w_{d+1}}g_{d+1}),f)$ as a function of $g_{d+1}$. Therefore $g^*$ also minimizes $\theta(g_{d+1})$. Bounding $\theta(g^*)$ by its mean we obtain

\begin{align*}
	\theta(g^*) \leq& 
	 \frac{1}{\#\mathcal{G}_{p,m-w_{d+1}}(f)} \sum_{h_{d+1} \in G_{p,m}\setminus \{0\}}r_p(h_{d+1},\gamma_{d+1}) \\ &\times\sum_{\bsh \in G_{p,m}^d}\left(\prod_{i=1}^dr_p(h_i,\gamma_i)\right) \sum_{\stackrel{g_{d+1}\in \mathcal{G}_{p,m-w_{d+1}}(f)}{\bs{h}\cdot \bsg \equiv -h_{d+1}x^{w_{d+1}}g_{d+1} \pmod{f}}} 1.
\end{align*}
Observe that $\gcd(f,h_{d+1}x^{w_{d+1}})=1$. Therefore the congruence $h_{d+1}x^{w_{d+1}}g_{d+1}\equiv -\bsh\cdot\bsg \pmod{f}$ has a unique solution in $G_{p,m}$ but not necessarily in $\mathcal{G}_{p,m-w_{d+1}}(f)$. In the case that $-\bsh\cdot\bsg \not\equiv 0 \pmod{f}$ we conclude that the congruence has at most one solution in $\mathcal{G}_{p,m-w_{d+1}}(f)$. If $-\bsh\cdot\bsg \equiv 0 \pmod{f}$ the congruence has no solution in $\mathcal{G}_{p,m-w_{d+1}}(f)$ since $0 \not \in \mathcal{G}_{p,m-w_{d+1}}(f)$. Hence we find by an application of \cite[Lemma 3.3]{DLP05}

\begin{align*}
	\theta(g^*) &\leq\frac{1}{\#\mathcal{G}_{p,m-w_{d+1}}(f)} \sum_{h_{d+1} \in G_{p,m}\back\{0\}}r_p(h_{d+1},\gamma_{d+1})\sum_{\bsh \in G_{p,m}^d}\prod_{i=1}^dr_p(h_i,\gamma_i)\\
	&=\frac{1}{\#\mathcal{G}_{p,m-w_{d+1}}(f)} \left[\prod_{i=1}^d\left(1+\gamma_i+\gamma_im\frac{p^2-1}{3p}\right)\right]\left(\gamma_{d+1}m\frac{p^2-1}{3p}\right).
\end{align*}

\noindent
By \eqref{thetakrilaineupoly} and the induction hypothesis we have that
\begin{align*}
R^{d+1}_{\gamma}&((\bsg,x^{w_{d+1}}g_{d+1}),f) =(1+\gamma_{d+1})R^{d}_{\gamma}(\bsg,f) + \theta(g_{d+1})\\
\leq& \frac{1}{p^m}\prod_{i=1}^d\left( 1+ \gamma_i+\gamma_ip^{\min\{w_{i},m\}}m\dfrac{p+1}{3}\right) \\ 
    &\times \left(1+\gamma_{d+1} + \gamma_{d+1}\frac{p^m}{\#\mathcal{G}_{p,m-w_{d+1}}(f)}m\frac{p^2-1}{3p}\right)\\
\leq& \frac{1}{p^m}\prod_{i=1}^{d+1}\left( 1+ \gamma_i+\gamma_ip^{\min\{w_{i},m\}}m\dfrac{p+1}{3}\right),
\end{align*}
where we used in the latter step that $\frac{p^m}{\#\mathcal{G}_{p,m-w_{d+1}}(f)} \leq \frac{p}{p-1}p^{\min\{w_{d+1},m\}}$. This follows from the fact that $\#\mathcal{G}_{p,m-w_{d+1}}(f) = p^{m-w_{d+1}}-1$ if $w_{d+1}<m$ and $\#\mathcal{G}_{p,m-w_{d+1}}(f)=1$ if $w_{d+1}\geq m$. This finishes the proof of Theorem \ref{R bound irreduciblekrilaineupoly}.
\end{proof}
As an immediate consequence of~\eqref{generalkrilaineupoly} and Theorem \ref{R bound irreduciblekrilaineupoly} we obtain the following result.
\begin{corollary}\label{cor:starDiskrilaineupoly}
Let $N = p^m$, $(w_j)_{j \geq 1}$ be a non-decreasing sequence of nonnegative integers and let $(g_1,\dots,g_s)\in \mathcal{G}_{p,m-\bsw}^s(f)$ for irreducible $f\in G_{p,m}$ be constructed using Algorithm \ref{reducedCbckrilaineupoly}. Then the polynomial lattice point set $\cP\left( (x^{w_1}g_1 , \dotsc , x^{w_s} g_s),f \right)$ has a weighted star discrepancy
\begin{align*}
D_{N,\bsgamma}^*&\left( (x^{w_1}g_1 , \dotsc , x^{w_s} g_s),f \right) \\
 &\leq \sum_{\substack{\mathfrak{u}\subseteq [s] \\ \mathfrak{u}\neq \emptyset}}{\gamma_{\mathfrak{u}}\left( 1- \left( 1 - \frac{1}{N} \right)^{|\mathfrak{u}|} \right)} + \frac{1}{N} \prod_{i=1}^{s}{\left( 1+\gamma_i + \gamma_i  p^{\min{\{ w_i, m \}}} m\frac{p+1}{3}  \right)}.
\end{align*}
\end{corollary}

\begin{remark}
Using the same argumentation as in Corollary~\ref{cor:tractabilitykrilaineupoly} we again obtain the sufficient condition $\sum\limits_{j=1}^\infty \gamma_j p^{w_j} < \infty$
for strong polynomial tractability and for the discrepancy bound $D_{N,\bsgamma}^*\left( (x^{w_1}g_1 , \dotsc , x^{w_s} g_s),f \right) = \mathcal{O}(N^{-1+\delta})$, with the implied constant independent of $s$, for any $\delta >0$.
\end{remark}

\vspace{0.5cm}
\noindent
\textbf{Acknowledgements.} We would like to thank Peter Kritzer and Friedrich Pillichs-hammer for their valuable comments and suggestions which helped to improve our paper.

\noindent{\bfseries Authors' Addresses:} \vspace{3mm}

\noindent Ralph Kritzinger, Mario Neum\"{u}ller, Institut f\"{u}r Finanzmathematik und Angewandte Zahlentheorie, Johannes Kepler Universit\"{a}t Linz, Altenbergerstr. 69, 4040 Linz, Austria. Email: ralph.kritzinger@jku.at, mario.neumueller@jku.at \vspace{3mm}

\noindent Helene Laimer, Johann Radon Institute for Computational and Applied Mathematics (RICAM),
Austrian Academy of Sciences, Altenbergerstr. 69, 4040 Linz, Austria. \\
          Email: helene.laimer@ricam.oeaw.ac.at

\end{document}